\documentclass[12pt]{amsart}
\usepackage{SSdefn}
\usepackage[boxsize=1em]{ytableau}

\newtheorem{maintheorem}{Theorem}

\newcommand{\gl}{\mathfrak{gl}}
\newcommand{\FI}{\mathbf{FI}}
\newcommand{\Pol}{\mathbf{Pol}}

\DeclareMathOperator{\maxdeg}{maxdeg}

\title{The spectrum of a twisted commutative algebra}

\author{Andrew Snowden}
\address{Department of Mathematics, University of Michigan, Ann Arbor, MI}
\email{\href{mailto:asnowden@umich.edu}{asnowden@umich.edu}}
\urladdr{\url{http://www-personal.umich.edu/~asnowden/}}

\thanks{AS was supported by NSF DMS-1453893.}

\date{\today}

\begin{document}

\begin{abstract}
A twisted commutative algebra is (for us) a commutative $\bQ$-algebra equipped with an action of the infinite general linear group. In such algebras the ``$\GL$-prime'' ideals assume the duties fulfilled by prime ideals in ordinary commutative algebra, and so it is crucial to understand them. Unfortunately, distinct $\GL$-primes can have the same radical, which obstructs one from studying them geometrically. We show that this problem can be eliminated by working with super vector spaces: doing so provides enough geometry to distinguish $\GL$-primes. This yields an effective method for analyzing $\GL$-primes.
\end{abstract}

\maketitle
\tableofcontents

\section{Introduction}

A \emph{twisted commutative algebra} (tca) is a commutative $\bQ$-algebra equipped with an action of the infinite general linear group $\GL_{\infty}$ under which it forms a polynomial representation; at least, that will be our definition for the moment. TCA's have been effectively used to study asymptotic problems in algebra (see, for example, \cite{fimodule,draisma-lason-leykin,stillman,delta-mod}), and are closely related to many particular objects of interest (such as EFW complexes \cite{efw}, determinantal varieties, and representations of infinite rank groups \cite{infrank}); moreover, all evidence so far points to a rich internal theory. It is therefore sensible to study these objects in more detail. While there have been many successes for particular tca's \cite{fimodule,LiRamos,sym2noeth, periplectic, symc1, symu1, symc1sp}, there has really only been one significant result to date for general tca's, namely, Draisma's topological noetherianity theorem \cite{draisma}. In this paper, we take another step towards understanding the general case: we largely solve the problem of understanding the equivariant prime ideals of tca's.

\subsection{Equivariant commutative algebra}

Let $A$ be a tca. One can then formulate equivariant analogs of many familiar concepts from commutative algebra\footnote{In fact, one can do this for commutative algebras in any tensor category.}:
\begin{itemize}
\item A \emph{$\GL$-ideal} of $A$ is an ideal that is $\GL$-stable.
\item A \emph{$\GL$-prime} is a $\GL$-ideal $\fp$ such that $VW \subset \fp$ implies $V \subset \fp$ or $W \subset \fp$, for subrepresentations $V,W \subset A$. Here $VW$ denotes the image of the map $V \otimes W \to A$.
\item The \emph{$\GL$-radical} of a $\GL$-ideal $I$, denoted $\rad_{\GL}{I}$, is the sum of all subrepresentations $V$ of $A$ such that $V^n \subset I$ for some $n$. Here $V^n$ denotes the image of the map $V^{\otimes n} \to A$. This is equal to the intersection of the $\GL$-primes containing $I$ (Proposition~\ref{prop:rad-int}).
\item The \emph{$\GL$-spectrum} of $A$, denoted $\Spec_{\GL}(A)$, is the set of $\GL$-primes, endowed with the usual Zariski topology. The closed subsets of $\Spec_{\GL}(A)$ correspond bijectively to $\GL$-radical ideals.
\end{itemize}
One can keep going, but this is all we need for the moment.

In ordinary commutative algebra, prime ideals are of central importance; this is no less true of $\GL$-primes in twisted commutative algebra. For instance: the support of an equivariant module is most naturally a subset of the $\GL$-spectrum; the $\GL$-primes can be used to generate the Grothendieck group of equivariant modules; and, under suitable hypotheses, one has an equivariant version of primary decomposition for $\GL$-ideals. Therefore, to understand tca's it is of crucial importance to understand their $\GL$-primes.

Every $\GL$-stable prime ideal of a tca is a $\GL$-prime. However, the converse is not true. The following is an instructive example:

\begin{example} \label{ex:main}
Let $A=\bigoplus_{n \ge 0} \lw^{2n}(\bQ^{\infty})$. Then $A$ is a tca, and a nice one at that: it is finitely generated and noetherian, in the equivariant sense. Suppose that $V$ and $W$ are non-zero subrepresentations of $A$. Since exterior powers are irreducible, it follows that $V$ contains $\lw^i(\bQ^{\infty})$ and $W$ contains $\lw^j(\bQ^{\infty})$ for some $i$ and $j$. Thus $VW$ contains $\lw^{i+j}(\bQ^{\infty})$. We have thus shown that if $V$ and $W$ are non-zero then so is $VW$. It follows that the zero ideal of $A$ is $\GL$-prime; in other words, $A$ is $\GL$-integral.
\end{example}

\noindent
This example is rather shocking when one first encounters it: every positive degree element of $A$ is nilpotent, and yet $A$ is a $\GL$-domain! This example shows that (from our current perspective) tca's do not have enough points to ``see'' their $\GL$-primes: indeed, $\Spec(A)$ is a single point, and thus cannot distinguish the two $\GL$-primes $(0)$ and $A_+$ of $A$. Thus $\GL$-domains appear to be divorced from geometry, which might diminish our hopes of understanding $\GL$-primes; fortunately, however, this appearance is deceiving.

\subsection{The key principle}

The category $\Rep^{\pol}(\GL_{\infty})$ of polynomial representations of $\GL_{\infty}$ is equivalent to the category $\Pol$ of polynomial functors of rational vector spaces; the equivalence is obtained by evaluating a functor on $\bQ^{\infty}$. We can thus view a tca $A$ as an algebra object of $\Pol$. From this perspective, $A(\bQ^{\infty})$ can be seen as the ``incarnation'' of $A$ in the category $\Rep^{\pol}(\GL_{\infty})$. However, polynomial functors can be evaluated on objects in any $\bQ$-linear tensor category. We can thus form the ``super incarnation'' $A(\bQ^{\infty|\infty})$ of $A$ by evaluating on the super vector space $\bQ^{\infty|\infty}$. This is an algebra object of $\Rep^{\pol}(\GL_{\infty|\infty})$. We can now succinctly express the point of this paper:

\vskip\baselineskip
\begin{center}
\begin{minipage}{30em}
\itshape\textbf{Key principle.} The geometry of the super incarnation of a tca is sufficiently rich to detect its $\GL$-primes.
\end{minipage}
\end{center}
\vskip.75\baselineskip

\noindent
This principle is borne out in the theorems stated below.

\begin{example}
Let $A$ be the tca from Example~\ref{ex:main}, regarded in $\Pol$. Then $A(\bQ^{0|s})=\bigoplus_{n \ge 0} \Sym^{2n}(\bQ^s)$ is the second Veronese subring of $\Sym(\bQ^s)$, which is a domain of Krull dimension $s$; in particular, its spectrum has plenty of points. For $r>0$, the algebra $A(\bQ^{r|s})$ is a nilpotent extension of $A(\bC^{0|s})$.
\end{example}

\subsection{Main results}
\label{ss:thms}

We now state several precise theorems. In what follows, $A$ is a tca (considered as a polynomial functor) and $I$ and $J$ are $\GL$-ideals of $A$.

\begin{maintheorem} \label{thm1}
We have $\rad_{\GL}{I} \subset \rad_{\GL}{J}$ if and only if $\rad{I(\bQ^{\infty|\infty})} \subset \rad{J(\bQ^{\infty|\infty})}$.
\end{maintheorem}

In other words, the theorem says that $\rad_{\GL}{I} \subset \rad_{\GL}{J}$ if and only if $\cV(J(\bQ^{\infty|\infty})) \subset \cV(I(\bQ^{\infty|\infty}))$; the latter condition is equivalent to $\cV(J(\bQ^{r|s})) \subset \cV(I(\bQ^{r|s}))$ for all $r$ and $s$ and thus (usually) reduces to a condition about finite dimensional algebraic varieties. Note that since we are only concerned with vanishing loci here, we can pass to the reduced quotient of $A(\bQ^{r|s})$, which is an ordinary (non-super) commutative ring. We emphasize that the theorem is false if one uses only ordinary vector spaces, as Example~\ref{ex:main} shows.

\begin{maintheorem} \label{thm2}
The ideal $\rad_{\GL}{I}$ is $\GL$-prime if and only if the ideal $\rad{I(\bQ^{\infty|\infty})}$ is prime.
\end{maintheorem}

Once again, we can verify the latter condition on finite dimensional spaces. The theorem therefore reduces the problem of showing that $\rad_{\GL}(I)$ is $\GL$-prime to showing that the algebraic varieties $\cV(I(\bQ^{r|s}))$ are irreducible.

In the two remaining theorems, we require a finiteness condition: we assume that $A_0$ is noetherian and that $A$ is finitely generated over $A_0$.

\begin{maintheorem} \label{thm3}
We have the following:
\begin{enumerate}
\item $A$ has finitely many minimal $\GL$-primes, say $\fp_1, \ldots, \fp_n$;
\item $A(\bQ^{\infty|\infty})$ has finitely many minimal primes, say $\fq_1, \ldots, \fq_m$;
\item $n=m$, and after applying a permutation we have $\fq_i=\rad{\fp_i(\bQ^{\infty|\infty})}$ for all $i$.
\end{enumerate}
\end{maintheorem}

This theorem gives a useful way to find the minimal $\GL$-primes, at least up to $\GL$-radical, and thus the irreducible components of the $\GL$-spectrum.

\begin{maintheorem} \label{thm4}
The $\GL$-spectrum $\Spec_{\GL}(A)$ of $A$ is a noetherian topological space.
\end{maintheorem}

This theorem is a strengthening of Draisma's topological noetherianity theorem: indeed, Draisma's theorem only encompasses the $\GL$-stable prime ideals, while this theorem accommodates all $\GL$-primes. In fact, this theorem is easily deduced by combining Draisma's theorem with our other theorems; we do not have a new proof of Draisma's result.

\subsection{An example}

In \S \ref{s:sym2}, we examine the tca $A=\Sym(\Sym^2(\bC^{\infty}))$. Using our main theorems, we classify the $\GL$-primes of $A$ (they are the rectangular ideals) and the $\GL$-radical ideals of $A$ (they are the ideals generated by a single irreducible representation). This example provides a good illustration of how our main theorems allow one to understand the equivariant commutative algebra of $A$ through standard geometric means. It also provides a reconceptualization of sorts for the rectangular ideals: they constitute the equivariant spectrum of $A$. These ideals have long been of interest, as they are symbolic powers of determinantal ideals.

\subsection{Connection to other work}

The notions of $\GL$-prime and the $\GL$-spectrum were discussed in \cite[\S 3]{symu1} (with slightly different terminology). However, that paper only works with so-called bounded tca's, and for these $\GL$-primes are the same as $\GL$-stable primes, so the chief difficulties disapper. (Note that the tca in Example~\ref{ex:main} is not bounded.)

In forthcoming joint work with Rohit Nagpal \cite{svar, sideals}, we study the $\fS_{\infty}$-primes in the ring $\bC[x_1, x_2, \ldots]$ and manage to completely classify them. (Here $\fS_{\infty}$ denotes the infinite symmetric group.) However, there is no general theory of $\fS_{\infty}$-primes yet.

One can, and usually does, define twisted commutative algebras without any reference to super vector spaces. However, our results show that one is essentially forced to use super vector spaces to fully understand tca's. This is reminscent of Deligne's theorem \cite[Theorem~0.6]{deligne}, in which a natural class of tensor categories (defined without any reference to super vector spaces) is characterized using super objects. It would be interesting to find a direct connection between these results.

\subsection{Outline}

In \S \ref{s:bg}, we review relevant background material. In \S \ref{s:key}, we prove the key theorem, which is a certain special case of Theorem~\ref{thm1}. In \S \ref{s:main}, we deduce the main theorems stated above. Finally, in \S \ref{s:sym2}, we work out the $\Sym(\Sym^2)$ case in detail.

\subsection*{Acknowledgments}

We thank Steven Sam for helpful discussions, and allowing us to include some jointly conceived ideas in \S \ref{s:sym2}.

\section{Background} \label{s:bg}

\subsection{Polynomial representations}

We recall some background material on polynomial representations. We refer to \cite{expos} and \cite[\S 2.2]{periplectic} for more details.

Regard $\GL_n$ as an algebraic group over $\bQ$. By an algebraic representation of $\GL_n$, we mean a comodule over $\bQ[\GL_n]$; the dimension is not required to be finite. We regard $\GL_{\infty}$ as the inductive system of algebraic groups $\bigcup_{n \ge 1} \GL_n$. An algebraic representation of $\GL_{\infty}$ is a vector space equipped with compatible algebraic representations of $\GL_n$ for all $n$. The basic example of such a representation is the standard representation $\bQ^{\infty}=\bigcup_{n \ge 1} \bQ^n$. A \emph{polynomial representation} of $\GL_{\infty}$ is one that occurs as a subquotient of a (possibly infinite) direct sum of tensor powers of the standard representation. We denote the category of such representations by $\Rep^{\pol}(\GL_{\infty})$. It is a $\bQ$-linear abelian category and closed under tensor products. The structure of this category is well-understood: it is semi-simple, and the simple objects are given by $\bS_{\lambda}(\bQ^{\infty})$, where $\bS_{\lambda}$ denotes the Schur functor associated to the partition $\lambda$. All simple objects are absolutely simple. Every polynomial representation $V$ admits a canonical decomposition $V = \bigoplus_{\lambda} V_{\lambda} \otimes \bS_{\lambda}(\bQ^{\infty})$ where the $V_{\lambda}$ are multiplicity spaces. We endow $V$ with a grading by declaring the elements of $\bS_{\lambda}(\bQ^{\infty})$ to have degree $\vert \lambda \vert$.

By an algebraic representation of $\GL_{r|s}$, we mean a comodule over the Hopf superalgebra $\bQ[\GL_{r|s}]$. We then define polynomial representations of $\GL_{\infty|\infty}$ just as before, and denote the category by $\Rep^{\pol}(\GL_{\infty|\infty})$. It is again semi-simple abelian and closed under tensor product, and the simple objects have the form $\bS_{\lambda}(\bQ^{\infty|\infty})$. Warning: if $V$ is a polynomial representation of $\GL_{\infty|\infty}$ then $V[1]$ (shift in super grading) is typically not a polynomial representation, according to our definition. This disagrees with the convention used in \cite{periplectic}.

Consider the category $\Fun(\Vec, \Vec)$ of endofunctors of the category $\Vec$ of $\bQ$-vector spaces. Let $T_n$ be the functor given by $T_n(V)=V^{\otimes n}$. We say that an object of $\Fun(\Vec, \Vec)$ is \emph{polynomial} if it is a subquotient of a (possibly infinite) direct sum of $T_n$'s. We let $\Pol$ be the category of polynomial functors. It is semi-simple abelian and closed under tensor products. The simple objects are the Schur functors $\bS_{\lambda}$. We have equivalences of categories
\begin{displaymath}
\xymatrix@R=1mm{
\Rep^{\pol}(\GL_{\infty}) & \Pol \ar[l] \ar[r] & \Rep^{\pol}(\GL_{\infty|\infty}) \\
F(\bQ^{\infty}) & F \ar@{|->}[r] \ar@{|->}[l] & F(\bQ^{\infty|\infty}) }
\end{displaymath}
These equivalences are compatible with tensor products, and respect algebras, modules, and ideals within the categories.

We will at times need to evaluate polynomial functors on finite dimensional spaces. We recall the relevant result.

\begin{proposition} \label{prop:schur-eval}
Let $\lambda$ be a partition, and let $r,s \in \bN$. If $\lambda_{r+1} \le s$ then $\bS_{\lambda}(\bQ^{r|s})$ is an absolutely irreducible representation of $\GL_{r|s}$; otherwise it vanishes. Moreover, the irreducible representations of $\GL_{r|s}$ obtained in this way are mutually non-isomorphic.
\end{proposition}

\begin{proof}
See \cite[\S 3.2.2]{chengwang}.
\end{proof}

We will require one additional simple result on polynomial representations. Define the \emph{width} of a partition $\lambda$, denote $w(\lambda)$, to be $\lambda_1$. Define the \emph{width} of a polynomial functor  $F$, denoted $w(F)$, to be the supremum of $w(\lambda)$ over those $\lambda$ for which $\bS_{\lambda}$ occurs in $F$.

\begin{proposition} \label{prop:width}
Let $F$ be a polynomial functor. The following are equivalent:
\begin{enumerate}
\item $w(F) \le N$.
\item For every $n$, every weight $\mu$ appearing in $F(\bQ^n)$ satisfies $\mu_1 \le N$.
\end{enumerate}
\end{proposition}

\begin{proof}
It suffices to treat the case where $F=\bS_{\lambda}$ is a simple object; we then have $w(F)=\lambda_1$. For $n \gg 0$, the $\GL_n$ representation $F(\bQ^n)$ is irreducible with highest weight $\lambda$. Thus there is a weight $\mu$ in $F(\bQ^n)$ with $\mu_1=w(F)$ (namely, $\mu=\lambda$). Furthermore, since $\lambda$ is a highest weight, we have $\mu_1 \le \lambda_1=w(F)$ for any other weight $\mu$. This proves the result.
\end{proof}

\subsection{Twisted commutative algebras}

A \emph{twisted commutative algebra} (tca) is a commutative algebra object in one of the three equivalent categories $\Rep^{\pol}(\GL_{\infty})$, $\Rep^{\pol}(\GL_{\infty|\infty})$, or $\Pol$. For the moment, we work in $\Rep^{\pol}(\GL_{\infty})$ to be definite.

Let $A$ be a tca in $\Rep^{\pol}(\GL_{\infty})$. Every polynomial representation carries a natural grading, and so $A$ is canonically graded. This is compatible with the ring structure, i.e., $A$ is a graded ring. In particular, we can regard $A$ as an algebra over is degree~0 piece $A_0$, which is an ordinary commutative ring.

By a ``subrepresentation'' of $A$, we mean a $\bQ$-subspace that is stable by $\GL_{\infty}$. In practice, $A$ will often be a $\bC$-algebra, but we use $\bQ$-subrepresentations nonetheless. One should think of a finite length subrepresentations of $A$ as providing an equivariant substitute for the concept of element (or perhaps finite sets of elements).

Suppose $A \to B$ is a homomorphism of tca's. We say that $B$ is \emph{finitely $\GL$-generated} over $A$ if there is some finite length subrepresentation $E$ of $B$ such that the natural map $A \otimes \Sym(E) \to B$ is surjective. We typically apply this in the case $A=B_0$.

\subsection{Ideals in tca's}

Let $A$ be as above, i.e., a tca in $\Rep^{\pol}(\GL_{\infty})$. A \emph{$\GL$-ideal} of $A$ is a $\GL$-stable ideal of $A$. We say that a $\GL$-ideal $I$ is \emph{finitely $\GL$-generated} if there is a finite length subrepresentation $E$ of $A$ that generates $I$ as an ideal. The sum and product of two finitely $\GL$-generated ideals is again finitely $\GL$-generated; for products, this relies on the fact that the tensor product of two finite length polynomial representations is again finite length.

We say that a $\GL$-ideal $\fp$ is \emph{$\GL$-prime} if $VW \subset \fp$ implies $V \subset \fp$ or $W \subset \fp$, for subrepresentations $V$ and $W$ of $A$; here $VW$ denotes the image of the map $V \otimes W \to A$. It is equivalent to ask the same condition with $V$ and $W$ finite length representations, or cyclic representations, or $\GL$-idelas, or finitely generated $\GL$-ideals. We define the \emph{$\GL$-radical} of a $\GL$-ideal $I$, denoted $\rad_{\GL}{I}$ to be the sum of all subrepresentations $V$ or $A$ such that $V^n \subset I$ for some $n$; again, one can use ideals in place of subrepresentations. We say that $I$ is \emph{$\GL$-radical} if $I=\rad_{\GL}{I}$. We note that every $\GL$-prime is $\GL$-radical.

\begin{remark}
A ``prime $\GL$-ideal'' of $A$ is a $\GL$-ideal of $A$ that is prime. This is potentially very different from a ``$\GL$-prime ideal'' of $A$. Similarly, ``radical $\GL$-ideal'' and ``$\GL$-radical ideal'' are potentially very different.
\end{remark}

We now establish some properties of the above definitions that are analogous to the classical situation.

\begin{proposition}
Let $I$ be a $\GL$-ideal of $A$ and let $E$ be a finite length subrepresentation of $\rad_{\GL}(I)$. Then $E^n \subset I$ for some $n$. Similarly, if $J$ is a finitely $\GL$-generated ideal contained in $\rad_{\GL}(I)$ then $J^n \subset I$ for some $n$.
\end{proposition}

\begin{proof}
By definition, we can write $\rad_{\GL}(I)=\sum_{i \in \cI} W_i$ where $W_i$ is a subrepresentation of $A$ such that $W_i^{n(i)} \subset I$ for some $n(i)$. Since $E$ is contained in $\rad(I)$ and of finite length, there is some finite subset $\cJ$ of $\cI$ such that $E \subset \sum_{i \in \cJ} W_i$. We thus have $E^n \subset I$ where $n=\# \cJ \cdot \max_{i \in \cJ} n(i)$. For the ideal case, simply pick a finite length subrepresentation that generates and appeal to the previous argument.
\end{proof}

\begin{proposition} \label{prop:rad-int}
Let $I$ be an ideal of $A$. Then $\rad_{\GL}(I)$ is the intersection of the $\GL$-primes containing $I$.
\end{proposition}

\begin{proof}
Let $\cP$ be the set of $\GL$-primes containing $I$. Suppose $\fp \in \cP$. Let $\fa$ be a finitely $\GL$-generated $\GL$-ideal contained in $\rad_{\GL}(I)$. Then $\fa^n \subset I \subset \fp$ for some $n$, and so $\fa \subset \fp$ since $\fp$ is $\GL$-prime. Since this holds for all $\fa$, we have $\rad_{\GL}(I) \subset \fp$. Since this holds for all $\fp$, we have $\rad_{\GL}(I) \subset \bigcap_{\fp \in \cP} \fp$.

We now prove the reverse inclusion. Let $\fc$ be a finitely $\GL$-generated $\GL$-ideal of $A$ not contained in $\rad(I)$. Let $S$ be the set of $\GL$-ideals $\fa$ of $A$ such that no power of $\fc$ is contained in $\fa$. Suppose that $\fa_1 \subset \fa_2 \subset \cdots$ is a chain in $S$, and let $\fa$ be its union. Then $\fa$ belongs to $S$ too. Indeed, if $\fc^n$ belongs to $\fa$ then, because it is finitely generated, it belongs to some $\fa_i$, a contradiction. Let $\fp$ be a maximal element of $S$, which exists by Zorn's lemma. We claim that $\fp$ is prime. Indeed, suppose $\fa \fb \subset \fp$, but $\fa,\fb \not\subset \fp$. Then $\fp+\fa$ and $\fp+\fb$ strictly contain $\fp$, and therefore do not belong to $S$. Thus $\fc^n \subset \fp+\fa$ and $\fc^m \subset \fp+\fb$ for some $n$ and $m$. Thus $\fc^{n+m} \subset (\fp+\fa)(\fp+\fb) \subset \fp+\fa \fb= \fp$, a contradiction. It follows that $\fc \not\subset \fp$, which completes the proof. Indeed, if $\bigcap_{\fp \in \cP} \fp$ were strictly larger than $\rad_{\GL}{I}$, then we could find a finitely $\GL$-generated $\fc$ contained in $\bigcap_{\fp \in \cP} \fp$ but not contained $\rad_{\GL}{I}$, and the above argument would yield a contradiction.
\end{proof}

\begin{proposition}
Every $\GL$-prime of $A$ contains some minimal $\GL$-prime of $A$.
\end{proposition}

\begin{proof}
An intersection of a descending chain of $\GL$-primes is clearly $\GL$-prime, so the claim follows from Zorn's lemma.
\end{proof}

The above concepts ($\GL$-prime, $\GL$-radical, etc.) are defined using only the language of the tensor category $\Rep^{\pol}(\GL)$. It follows that the same definitions can be made in $\Rep^{\pol}(\GL_{\infty|\infty})$ and $\Pol$, and that the definitions agree on objects that correspond under the equivalences. Thus the above propositions also hold in all three settings. In fact, one can formulate and prove these results for commutative algebra objects in quite general tensor categories.

One important construction that \emph{cannot} be formulated using only the language of the tensor category is the ordinary radical. Suppose $A$ is a tca in $\Pol$ and $I$ is a $\GL$-ideal in it. We can then consider their incarnations $I(\bQ^{\infty}) \subset A(\bQ^{\infty})$ in $\Rep^{\pol}(\GL_{\infty})$, and form $\rad(I(\bQ^{\infty}))$. Similarly, we can consider their incarnations in $\Rep^{\pol}(\GL_{\infty|\infty})$ and form $\rad(I(\bQ^{\infty|\infty}))$. There is no reason to expect these two radicals to be comparable in any way (except in tautological ways, e.g., both contain $I$). In fact, the point of this paper is that they really are not comparable, and the construction is better behaved on the super side.

\subsection{Minimal primes}

We require the following result.

\begin{proposition} \label{prop:min-prime}
Let $A$ be a tca in $\Rep^{\pol}(\GL_{\infty})$ and let $\fp$ be a minimal prime of $A$. Then $\fp$ is $\GL$-stable.
\end{proposition}

\begin{proof}
Consider the maps
\begin{displaymath}
\xymatrix@C=1.5cm{
A \ar[r]^-{\Delta} &
A \otimes \bQ[\GL_n] \ar[r]^-{\pi \otimes \id} &
A/\fp \otimes \bQ[\GL_n] \ar[r]^-{\id \otimes \epsilon} &
A/\fp }
\end{displaymath}
where $\Delta$ is comultiplication, $\pi \colon A \to A/\fp$ is the quotient map, and $\epsilon \colon \bQ[\GL_n] \to \bQ$ is the counit. The composition is equal to $\pi$ by the axioms for a comodule, and thus has kernel $\fp$. We thus see that $\fq=\ker((\pi \otimes \id) \circ \Delta) \subset \fp$. However, $\fp$ is prime and $\bQ[\GL_n]$ is a localization of a polynomial algebra over $\bQ$, and so $A/\fp \otimes \bQ[\GL_n]$ is a domain. Thus $\fq$ is prime. Since $\fp$ is minimal, we must have $\fp=\fq$. Now, let $x$ be an element of $\fp$, and write $\Delta(x) = \sum_{i=1}^n a_i \otimes b_i$ where $a_i \in A$ and $b_i \in \bQ[\GL_n]$ are $\bQ$-linearly independent elements. Then $0=(\pi \otimes \id)(\Delta(x))=\sum_{i=1}^n \pi(a_i) \otimes b_i$. Since the $b_i$ are linearly independent, it follows that $\pi(a_i)=0$ for all $i$, and so $a_i \in \fp$ for all $i$. Thus $\Delta(\fp) \subset \fp \otimes \bQ[\GL_n]$, and so $\fp$ is $\GL_n$-stable. Since this holds for all $n$, it follows that $\fp$ is $\GL_{\infty}$-stable.
\end{proof}

\begin{remark}
The analog of this statement for tca's in $\Rep^{\pol}(\GL_{\infty|\infty})$ does not hold: the above proof fails since $\bQ[\GL_{r|s}]$ is not a domain.
\end{remark}

\subsection{Radicals of $\GL$-primes}

We require the following result on $\GL$-primes. See \cite[\S 8.6]{expos} for some similar results.

\begin{proposition} \label{prop:rad-prime}
Let $A$ be a tca in $\Rep^{\pol}(\GL_{\infty|\infty})$ and let $\fp$ be a $\GL$-prime of $A$. Then $\rad(\fp)$ is prime.
\end{proposition}

We require some preliminary work before proving the proposition. Let $\{e_i\}_{i \in \cI}$ be a homogeneous basis for $\bV=\bQ^{\infty|\infty}$. Given an element $x$ in a polynomial representation $V$ of $\GL_{\infty|\infty}$ and a subset $S$ of $\cI$, we say that $x$ has \emph{support contained in $S$} if $V$ can be embedded into a direct sum of tensor powers of $\bQ^{\infty|\infty}$ such that $x$ can be expressed using the basis elements in $\cI$. One can define this more canonically by looking at the weight decomposition of $x$. We define the \emph{support} of an element $X$ of $\gl_{\infty|\infty}$ to be the set of indices $i$ such that $X$ has a non-zero entry in row $i$ or column $i$. We say that an element of $\cU(\fgl_{\infty|\infty})$ has \emph{support contained in $S$} if it can be expressed in terms of elements of $\fgl_{\infty|\infty}$ having this property. We say that elements of representations or $\cU(\fgl_{\infty|\infty})$ are \emph{disjoint} if they have disjoint supports (or if they have supports contained in disjoint sets). For an element $x \in A$, we let $\langle x \rangle$ be the $\GL$-ideal it generates.

\begin{lemma} \label{lem:rad-prime-3}
Suppose $x,y \in A$ are disjoint elements such that $xy=0$. Then $\langle x \rangle \cdot \langle y \rangle = 0$.
\end{lemma}

\begin{proof}
Let $y'$ be an element of $\langle y \rangle$ that is disjoint from $x$. Write $y'=ay$ where $a \in \cU(\fgl_{\infty|\infty})$. Now, the support of $a$ may overlap with that of $x$, that is, we may use auxiliary basis vectors in the process of building $y'$ from $y$. However, it does not matter which auxiliary basis vectors we use, so we can modify $a$ if necessary so that it  is disjoint from $x$. More rigorously, choose a permutation $\sigma$ of $\cI$ that fixes the supports of $y$ and $y'$, and such that $\sigma a \sigma^{-1}$ is disjoint from $x$. Then $y'=\sigma a \sigma^{-1} y$ and $\sigma a \sigma^{-1}$ is disjoint from $x$. Now, applying $a$ to the expression $xy=0$, and using the fact that $a$ commutes with $x$ since it is disjoint from $x$, we find $x (ay)=0$, that is, $xy'=0$.

Now let $y'$ be an arbitrary element of $V$. We can then write $y'=\sigma y''$ where $\sigma$ is a permutation of $\cI$ and $y''$ is disjoint from $x$. Let $E \in \gl_{\infty} \times \gl_{\infty}$ act by $\sigma$ on the support of $y''$ and~0 on the remaining basis vectors. Then $Ex=0$ and $Ey''=y'$. Since $y''$ is disjoint from $x$, we have $x y''=0$ by the previous paragraph. Applying $E$ to this equation gives $x y'=0$. This completes the proof.
\end{proof}

\begin{lemma} \label{lem:rad-prime-4}
Let $x,y \in A$ be super homogeneous elements satisfying $xy=0$. Then there exists $n \ge 0$ such that $\langle x^n \rangle \cdot \langle y \rangle =0$.
\end{lemma}

\begin{proof}
We claim that for any $a \in \cU(\gl_{\infty|\infty})$ there exists $n \ge 0$ such that $x^n \cdot ay=0$. This is clear for $a=1$. Suppose now it is true for $a$, and let us prove it for $Ea$, with $E \in \gl_{\infty|\infty}$. It suffices to treat the case where $E$ is super homogeneous. Let $n$ be such that $x^n \cdot ay=0$. Applying $E$, we find $n x^{n-1} Ex \cdot ay \pm x^n \cdot Eay=0$, where the sign depends on super degrees. Multiplying by $x$ and using the fact that $x^n \cdot ay=0$, we find $x^{n+1} \cdot Eay=0$. The claim now follows.

Now let $y' \in \langle y \rangle$ be disjoint from $x$ and generate $\langle y \rangle$; for instance, one could take $y'=\sigma y$ for an appropriate permutation $\sigma$ of $\cI$. Since $y'=ay$ for some $a \in \cU(\fgl_{\infty|\infty})$, the previous paragraph gives $x^n y'=0$ for some $n$. By Lemma~\ref{lem:rad-prime-3}, we find $\langle x^n \rangle \cdot \langle y \rangle=0$, and so the result follows.
\end{proof}

\begin{proof}[Proof of Proposition~\ref{prop:rad-prime}]
Passing to $A/\fp$, we assume $\fp=0$ is $\GL$-prime. We must show that $\rad(A)$ is prime. Since all odd elements of $A$ are nilpotent, we have $A_1 \subset \rad(A)$. It thus suffices to show that if $xy \in \rad(A)$ with $x$ and $y$ even then $x \in \rad(A)$ or $y \in \rad(A)$. Thus let even elements $x$ and $y$ be given such that $xy$ is nilpotent, say $(xy)^k=0$. Since $x$ and $y$ are even, they commute, and so $x^ky^k=0$. By Lemma~\ref{lem:rad-prime-4}, there exists $n \ge 0$ such that $\langle x^{nk} \rangle \cdot \langle y^k \rangle =0$. Since $(0)$ is $\GL$-prime, it follows that $x^{nk}=0$ or $y^k=0$. Thus either $x$ or $y$ is nilpotent, which completes the proof.
\end{proof}

\subsection{Draisma's theorem}

Suppose that a group $G$ acts on a topological space $X$. We say that $X$ is \emph{$G$-noetherian} if every descending chain of closed $G$-stable subsets stabilizes. Draisma \cite[Corollary~3]{draisma} proved the following important theorem in this context:

\begin{theorem} \label{thm:draisma}
Let $A$ be a tca in $\Rep^{\pol}(\GL_{\infty})$ such that $A_0$ is noetherian and $A$ is finitely generated over $A_0$. Then $\Spec(A)$ is $\GL_{\infty}$-noetherian.
\end{theorem}

In fact, Draisma only states this theorem when $A_0$ is finitely generated over a field, but a slight modification in his proof yields the above statement. Since it is not critical for this paper, we do not include details. We give a few corollaries of the theorem.

\begin{corollary} \label{cor:acc-rad}
Let $A$ be as in Theorem~\ref{thm:draisma}. Then every ascending chain of radical $\GL$-ideals in $A$ stabilizes.
\end{corollary}

\begin{proof}
This follows since radical $\GL$-ideals of $A$ correspond bijectively to $\GL$-stable closed subsets of $\Spec(A)$.
\end{proof}

\begin{corollary} \label{cor:GL-fixed-space}
Let $A$ be as in Theorem~\ref{thm:draisma}. Let $Y=\Spec(A)$ and let $X=Y^{\GL}$ be the subset consisting of $\GL_{\infty}$-stable prime ideals, endowed with the subspace topology. Then $X$ is a noetherian spectral space.
\end{corollary}

\begin{proof}
Suppose that $Z$ is a closed subset of $X$. Then $Z$ has the form $W \cap X$ for some closed subset $W$ of $Y$. Since every point in $Z$ is $\GL_{\infty}$-invariant, it follows that $Z=(gW) \cap X$ for any $g \in \GL_{\infty}$. Thus $Z=W' \cap X$ where $W'=\bigcap_{g \in \GL_{\infty}} gW$ is a $\GL$-stable closed subset of $Y$. It follows that $Z=\ol{Z} \cap X$, where $\ol{Z}$ is the closure of $Z$ in $Y$, since $Z \subset \ol{Z} \subset W'$.

Now suppose that $\cdots \subset Z_2 \subset Z_1$ is a descending chain of closed sets in $X$. Then $\cdots \subset \ol{Z}_2 \subset \ol{Z}_1$ is a descending chain of $\GL$-stable closed subsets of $Y$, and thus stabilizes by Theorem~\ref{thm:draisma}. Since $Z_i=\ol{Z}_i \cap X$, it follows that the original chain stabilizes too. Thus $X$ is noetherian.

Finally, let $Z$ be an irreducible closed subset of $X$. Then $\ol{Z}$ is a $\GL$-stable irreducible closed subset of $Y$. Its generic point is thus $\GL$-stable, and therefore belongs to $X$. One easily sees that it is the unique generic point for $Z$. Thus $X$ is sober. Since it is also noetherian, it is spectral.
\end{proof}

\begin{corollary} \label{cor:fin-min-prime}
Let $A$ be as in Theorem~\ref{thm:draisma}. Then $A$ has finitely many minimal primes.
\end{corollary}

\begin{proof}
Let $X$ be as in Corollary~\ref{cor:GL-fixed-space}. Since $X$ is noetherian, it has finitely many irreducible components. Since it is sober, these components correspond to the minimal $\GL$-stable prime ideals of $X$. There are thus finitely many of these. However, if $\fp$ is any minimal prime then it is $\GL$-stable by Proposition~\ref{prop:min-prime}, and thus obviously a minimal $\GL$-stable prime. The result follows.
\end{proof}

\section{The key result} \label{s:key}

The following is the key theorem of this paper: it is the bridge that connects equivariant concepts to ordinary ones.

\begin{theorem} \label{thm:nil}
Let $A$ be a tca in $\Pol$ that is finitely generated over $\bQ$. The following are equivalent:
\begin{enumerate}
\item Every positive degree homogeneous element of $A(\bQ^{r|s})$ is nilpotent, for all $r$, $s$.
\item The ideal $A_+$ is nilpotent.
\end{enumerate}
\end{theorem}

\begin{proof}
It is clear that (b) implies (a). We prove the converse. We proceed by induction on the degree of generation of $A$. Thus suppose $A$ is generated in degrees $\le d$ and satisfies (a), and that the theorem is true for tca's generated in degrees $<d$.

Before getting into the argument, we introduce a piece of notation. For a polynomial functor $F$, we let $F'$ be the polynomial functor defined by $F'(V)=F(\bQ \oplus V)$. We note that $F'$ carries an action of $\bG_m$, through its action on $\bQ$. If $x \in F'(\bQ^{r|s})=F(\bQ^{r+1|s})$ is a weight vector for $\GL_{r|s}$ of weight $(a_1, \ldots, a_r; b_1, \ldots, b_s)$ and simultaneously a weight vector for $\bG_m$ of weight $k$ then it is also a weight vector for $\GL_{r+1|s}$ of weight $(k, a_1, \ldots, a_r; b_1, \ldots, b_s)$.

Let $E \subset A_0 \oplus \cdots \oplus A_d$ be a subrepresentation generating $A$. Let $E'=E'_0 \oplus \cdots \oplus E'_d$ be the weight space decomposition for $E'$ with respect to the $\bG_m$ action; note that $E'_i$ is a polynomial functor of degree $\le d-i$. Let $B$ be the subalgebra of $A'$ generated by $E'_1, \ldots, E'_d$, and let $C$ be the subalgebra generated by $E'_0$. For $i>0$, every weight vector of $E'_i(\bQ^{r|s})$ is a weight vector of $E(\bQ^{r+1|s})$ of non-zero weight (since the first component of the weight is $i$), and thus a positive degree homogenous element of $A(\bQ^{r+1|s})$, and thus nilpotent. We thus see that the generators of $B(\bQ^{r|s})$ are nilpotent, and so $B$ satisfies (a). Since $B$ is finitely generated in degrees $<d$, we can apply the inductive hypothesis to conclude that $B_n=0$ for $n \gg 0$. Since the degree~0 generators of $B$ are also nilpotent, it follows that $B$ has finite length as a polynomial functor. In particular, only finitely many $\bG_m$ weights appear in $B$; say that the largest one is $N$.

Let $T$ be the maximal torus of $\GL_{n+1}$. Suppose that $\lambda=(\lambda_1, \ldots, \lambda_{n+1})$ is a weight of $T$ that appears in $A(\bC^{n+1})=A'(\bQ^n)$; we note that $\lambda_1$ records the action of $\bG_m$. Since $B(\bQ^n)$ and $C(\bQ^n)$ generate $A'(\bQ^n)$ are are $T$-stable, it follows that $\lambda$ can be written in the form $\mu+\nu$ where $\mu$ is a weight of $T$ appearing in $B(\bQ^n)$ and $\nu$ is one in $C(\bC^n)$. We have $\mu_1 \le N$ by the definition of $N$. Since $C$ is generated by $E'_0$, on which $\bG_m$ acts trivially, we see that $\nu_1=0$. Thus $\lambda_1 \le N$. Since this holds for all weights in $A(\bQ^{n+1})$ for any $n$, it follows that $A$ has width $\le N$ by Proposition~\ref{prop:width}.

Decompose $A$ as $\bigoplus_{\lambda} A_{\lambda} \otimes \bS_{\lambda}$ where $A_{\lambda}$ is a multiplicity space. We have just shown that $A_{\lambda}$ is only non-zero when $\lambda_1 \le N$. Consider the superalgebra $A(\bQ^{0|N}) = \bigoplus_{\lambda} A_{\lambda} \otimes \bS_{\lambda}(\bQ^{0|N})$. This is finitely generated, and every positive degree element is nilpotent by assumption. Thus $A(\bQ^{0|N})_n=0$ for $n$ sufficiently large, say $n>M$. We thus have $A_{\lambda} \otimes \bS_{\lambda}(\bQ^{0|N})=0$ for $\vert \lambda \vert>M$, and so $A_{\lambda}=0$: indeed, we know this already if $\lambda_1>N$, and otherwise $\bS_{\lambda}(\bQ^{0|N})$ is non-zero by Proposition~\ref{prop:schur-eval}. It follows that $A_n=0$ for $n>M$ as well, and so $A$ satisfies (b). This proves the theorem.
\end{proof}

\begin{example}
Suppose $A$ is generated over $\bQ$ by $E=\Sym^2$. The space $E'(\bQ^n)=E(\bQ^{n+1})$ has for a basis elements $x_{i,j}$ with $1 \le i \le j \le n+1$. The degree of $x_{i,j}$ under the $\bG_m$ action is simply the number of indices equal to~1. Thus $E'_2(\bQ^n)$ is spanned by $x_{1,1}$, while $E'_1(\bQ^n)$ is spanned by the $x_{1,j}$ with $2 \le j$, and $E'_0(\bQ^n)$ is spanned by the $x_{i,j}$ with $2 \le i,j$. Thus $B(\bQ^n)$ is generated by $x_{1,1}$, which is $\GL_n$ invariant, and the $x_{1,j}$ with $2 \le j \le n+1$, which generate a copy of the standard representation of $\GL_n$. We thus see that $B$ is generated by $\Sym^0 \oplus \Sym^1$, which has degree $\le 1$ as a polynomial functor. This shows that $B$ can have degree~0 generators even if $A$ does not. The algebra $C(\bQ^n)$ is generated by the $x_{i,j}$ with $2 \le i,j$, and so $\bG_m$ acts trivially on it.
\end{example}

\begin{corollary} \label{cor:nil}
Let $A$ be an arbitrary tca in $\Pol$ and let $E$ be a finite length subrepresentation of $A$. The following are equivalent:
\begin{enumerate}
\item Every element of $E(\bQ^{r|s})$ is nilpotent, for all $r$ and $s$.
\item The space $E$ is nilpotent, i.e., the map $E^{\otimes n} \to A$ is zero for some $n$.
\end{enumerate}
\end{corollary}

\begin{proof}
Obviously (b) implies (a); we prove the converse. First suppose that $E$ is generated in positive degrees. Let $B$ be the sub tca of $A$ generated over $\bQ$ by $E$. Then $B(\bQ^{r|s})$ is generated as a subalgebra of $A(\bQ^{r|s})$ by $E(\bQ^{r|s})$, and so every positive degree element is nilpotent. Thus, by Theorem~\ref{thm:nil}, we see that $B_+$ is nilpotent. Since $E \subset B_+$, we see that it too is nilpotent.

In general, write $E=E_0 \oplus E_+$ where $E_0$ is the degree~0 piece of $E$ and $E_+$ is the sum of the positive degree pieces of $E$. Then $E_0$ is finite dimensional and every element is nilpotent, so $E_0$ is nilpotent, and $E_+$ is nilpotent by the previous paragraph. Thus $E$ is nilpotent.
\end{proof}

The proof of Theorem~\ref{thm:nil} is effective, in the following sense. For a finite length polynomial functor $E$, let $\cS_k(E)$ be the class of tca's $A$ such that (i) $A$ contains a copy of $E$ that generates it over $\bQ$; and (ii) each weight space of $E(\bC^{r|s})$ of non-zero weight admits a basis consisting of $k$-nilpotent elements, for any $r$ and $s$. Let $\eta_k(E)$ be the supremum of $\maxdeg(A)$ over $A \in \cS_k(E)$, where $\maxdeg$ denotes the maximum non-zero degree. Theorem~\ref{thm:nil} simply states that $\maxdeg(A)$ is finite for $A \in \cS_k(E)$. In fact, the proof yields a bound on $\eta_k(E)$. Let $d$ be the degree of $E$, let $E'_i$ be as in the proof of the theorem, and let $P$ be the polynomial defined by $P(n)=\dim{E(\bQ^{0|n})}$. Examining the proof, one finds
\begin{displaymath}
\eta_k(E) \le kd \cdot P(d\eta_k(E'_1 \oplus \cdots \oplus E'_d)+kd\dim{E(\bC)}).
\end{displaymath}
This allows one to inductively obtain a bound on $\eta_k(E)$ since the argument to $\eta_k$ on the right has smaller degree than $E$. Making the rough approximation $P(x) \approx x^d$, one finds
\begin{displaymath}
\eta_k(E) \lessapprox \prod_{i=1}^d (ik)^{d!/i!}.
\end{displaymath}
This upper bound is quite large; e.g., it is substantially larger than $k^{d!} 2^{2^d}$. We do not know how close it is to the true behavior of $\eta_k(E)$.

\section{The main theorems} \label{s:main}

We fix a tca $A$ in $\Rep^{\pol}(\GL_{\infty|\infty})$ for this section. Consider the following condition:
\begin{itemize}
\item[($\ast$)] $A_0$ is noetherian and $A$ is finitely generated over $A_0$.
\end{itemize}
We will sometimes require this condition, and sometimes not. Our goal now is to prove the main theorems stated in \S \ref{ss:thms}.

\begin{proposition}[Theorem~\ref{thm1}] \label{prop:thm1}
Let $I$ and $J$ be $\GL$-ideals of $A$. Then $I \subset \rad_{\GL}{J}$ if and only if $I \subset \rad{J}$.
\end{proposition}

Note that $I \subset \rad_{\GL}{J}$ if and only if $\rad_{\GL}{I} \subset \rad_{\GL}{J}$, and similarly for ordinary radicals, so this proposition is indeed equivalent to Theorem~\ref{thm1}.

\begin{proof}
We may replace $I$ with $I+J$ without changing either condition. We may then check the conditions after passing to $A/J$. Thus we may simply assume from the outset that $J=0$.

If $I \subset \rad_{\GL}(A)$ then $I \subset \rad(A)$, since we have a containment $\rad_{\GL}(A) \subset \rad(A)$. Conversely, suppose that $I \subset \rad(A)$. Let $E$ be a finite length subrepresentation of $I$. Then every element of $E$ is nilpotent, and so $E$ is nilpotent by Corollary~\ref{cor:nil}. Thus $E \subset \rad_{\GL}(A)$. Since this holds for all $E$, it follows that $I \subset \rad_{\GL}(A)$.
\end{proof}

We now introduce an auxiliary algebra that we will be helpful in what follows. Let $B=A/\rad(A)$. The ideal $\rad(B)$ is typically not $\GL_{\infty|\infty}$ stable, but is clearly stable by $\GL_{\infty} \times \GL_{\infty} \subset \GL_{\infty|\infty}$. It is therefore also stable by the diagonal subgroup $\GL_{\infty} \subset \GL_{\infty} \times \GL_{\infty}$, and so this acts on $B$. Choose a surjection $A_0 \otimes \Sym(V) \to A$, for some representation $V \subset A$. The restriction of $V$ to the diagonal $\GL_{\infty}$ has the form $W_0 \oplus W_1[1]$, where $W_0$ and $W_1$ are polynomial representations of $\GL_{\infty}$, and $[1]$ indicates the super grading. It follows that $A$ is a quotient of $A_0 \otimes \Sym(W_0 \oplus W_1[1])$. Since $B$ has no odd part, we see that it is a quotient of $A_0 \otimes \Sym(W_0)$. In other words, $B$ is a twisted commutative algebra in the category $\Rep^{\pol}(\GL_{\infty})$. If $A$ satisfies $(\ast)$ then we can take $V$ to be a finite length representation. It follows from basic properties of Schur functors that $W_0$ and $W_1$ are then of finite length as well, and so $B$ is finitely generated over $B_0=A_0$.

\begin{proposition}[Theorem~\ref{thm4}] \label{prop:thm4}
Suppose $(\ast)$ holds. Then $\Spec_{\GL}(A)$ is noetherian.
\end{proposition}

\begin{proof}
It suffices to show that every ascending chain of $\GL$-radical ideals in $A$ stabilizes. Thus let $I_1 \subset I_2 \subset \cdots$ be such a chain. Then $\rad(I_1) \subset \rad(I_2)$ is an ascending chain of $\GL_{\infty}$-stable radical ideals of $A$, and thus corresponds to an ascending chain of $\GL_{\infty}$-stable radical ideals of $B$. It therefore stabilizes by Corollary~\ref{cor:acc-rad}. By Proposition~\ref{prop:thm1}, it follows that the original chain stabilizes.
\end{proof}

\begin{proposition}[Theorem~\ref{thm3}] \label{prop:min}
Suppose $(\ast)$ holds.
\begin{enumerate}
\item $A$ has finitely many minimal $\GL$-primes, say $\fp_1, \ldots, \fp_n$;
\item $A$ has finitely many minimal primes, say $\fq_1, \ldots, \fq_m$;
\item $n=m$, and after applying a permutation we have $\fq_i=\rad(\fp_i)$ for all $i$.
\end{enumerate}
\end{proposition}

\begin{proof}
(a) This is an immediate consequence of Proposition~\ref{prop:thm4}.

(b) The minimal primes of $A$ correspond bijectively to those of $B$, and there are finitely many of these by Corollary~\ref{cor:fin-min-prime}.

(c) Consider a minimal prime $\fq_i$. We have
\begin{displaymath}
\fp_1 \cap \cdots \cap \fp_n = \rad_{\GL}(A) \subset \rad(A) \subset \fq_i.
\end{displaymath}
Since $\fq_i$ is prime, it follows that $\fp_j \subset \fq_i$ for some $j$. Thus $\rad(\fp_j) \subset \fq_i$. Since $\rad(\fp_j)$ is prime (Proposition~\ref{prop:rad-prime}) and $\fq_i$ is a minimal prime, we have $\rad(\fp_j)=\fq_i$. By Proposition~\ref{prop:thm1}, it follows that $j$ is unique: indeed, if $\rad(\fp_j)=\rad(\fp_k)$ then $\fp_j=\fp_k$ and so $j=k$.

To complete the proof, it suffices to show that $\rad(\fp_j)$ is a minimal prime for all $j$. We know that $\rad(\fp_j)$ is prime. It therefore contains some minimal prime $\fq_i$. We have shown that $\fq_i=\rad(\fp_k)$ for some $k$. Thus $\fp_k \subset \rad(\fp_j)$, and so $\fp_k \subset \rad_{\GL}(\fp_j)=\fp_j$ by Proposition~\ref{prop:thm1}. Since $\fp_j$ is a minimal $\GL$-prime, it follows that $\fp_j=\fp_k$. Hence $\rad(\fp_j)=\fq_i$ is a minimal prime.
\end{proof}

\begin{proposition}[Theorem~\ref{thm2}]
Let $I$ be a $\GL$-ideal of $A$. Then $\rad_{\GL}(I)$ is $\GL$-prime if and only if $\rad(I)$ is prime.
\end{proposition}

\begin{proof}
If $I$ is $\GL$-prime then $\rad(I)$ is prime by Proposition~\ref{prop:rad-prime}. Now suppose that $\rad(I)$ is prime and $A$ satisfies $(\ast)$. Then $A/I$ has a unique minimal prime, and therefore a unique minimal $\GL$-prime by Proposition~\ref{prop:min}. Thus there is a unique minimal $\GL$-prime $\fp$ over $I$, and so $\rad_{\GL}(I)=\fp$ is $\GL$-prime. Finally, suppose that $\rad(I)$ is prime and $A$ is arbitrary. Write $A=\bigcup A_i$ where $\{A_i\}$ is a directed family of sub tca's satisfying $(\ast)$. Then $\rad(I \cap A_i) = \rad(I) \cap A_i$ is prime, and so $\rad_{\GL}(I \cap A_i)=\rad_{\GL}(I) \cap A_i$ is $\GL$-prime by the previous case. Since $\rad_{\GL}(I) \cap A_i$ is prime for all $i$, it follows that $\rad_{\GL}(I)$ is prime by Lemma~\ref{lem:prime-limit} below.
\end{proof}

\begin{lemma} \label{lem:prime-limit}
Let $A$ be a tca, and suppose that $A=\bigcup A_i$ for some directed family $\{A_i\}$ of sub tca's. Let $\fp$ be a $\GL$-ideal of $A$. If $\fp \cap A_i$ is $\GL$-prime for all $i$ then $\fp$ is $\GL$-prime.
\end{lemma}

\begin{proof}
Let $V$ and $W$ be finite length subrepresentations of $A$ such that $VW \subset \fp$. Since $V$ adn $W$ are finite length, there is some $i$ such that $V$ and $W$ are contained in $A_i$. Thus $VW \subset \fp \cap A_i$. Since $\fp \cap A_i$ is $\GL$-prime, it follows that $V \subset \fp \cap A_i$ or $W \subset \fp \cap A_i$. Thus $V \subset \fp$ or $W \subset \fp$, and so $\fp$ is $\GL$-prime.
\end{proof}

\section{An example}
\label{s:sym2}

Let $A$ be the tca in $\Pol$ given by $A(V)=\bC \otimes \Sym(\Sym^2(V))$. Our goal is to classify the $\GL$-prime and $\GL$-radical ideals of $A$.

\subsection{The ideal lattice of $A$}

The decomposition of $A$ into irreducibles is well-known:
\begin{displaymath}
A = \bigoplus \bS_{2\lambda},
\end{displaymath}
where the sum is over all partitions $\lambda$, and $2\lambda=(2\lambda_1, 2\lambda_2, \ldots)$. See, for example,  \cite[\S I.5, Example 5]{macdonald}. For a partition $\lambda$, let $I_{\lambda}$ be the ideal of $A$ generated by $\bS_{2\lambda}$. The following result determines the ideal structure of $A$:

\begin{proposition} \label{prop:sym2ideals}
The ideal $I_{\lambda}$ is the sum of those $\bS_{2\mu}$ for which $\lambda \subset \mu$. In particular, $I_{\lambda} \subset I_{\mu}$ if and only if $\mu \subset \lambda$.
\end{proposition}

\begin{proof}
This was originally proved in \cite{abeasis}, but that is a difficult reference to obtain. The analogous result for $\Sym(\lw^2)$ is proved in \cite[Theorem~3.1]{pfaffians}. That case actually implies this one, since $\Sym(\Sym^2(V))=\Sym(\lw^2(V[1]))$. A complete proof in the case where $\lambda$ is a rectangle also appears in \cite[Corollary~2.8]{periplectic}. A closely related result appears in \cite[Theorem~4.1]{CEP}.
\end{proof}

The ideals generated by rectangular shapes will be particularly important, so we introduce some notation for them. We let $\rho(r,s)$ be the partition with $r$ rows each of length $s$; thus the Young diagram for $\rho(r,s)$ is an $r \times s$ rectangle. We let $I_{r,s}=I_{\rho(r,s)}$. If $r=0$ or $s=0$ then $\rho(r,s)$ is an empty partition and $I_{r,s}$ is the unit ideal.

Let $\lambda$ be a partition. By a \emph{corner} of $\lambda$ we mean a pair $(r,s)$ such that $\lambda$ has a box in the $r$th row and $s$th column, but no box below or to the right of this one. For example, in the following Young diagram the corners have been shaded:
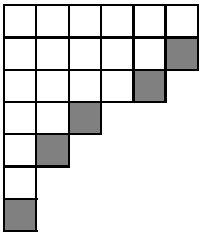
\begin{displaymath}
\ydiagram[*(white)]{6,5,4,2,1,1}*[*(gray)]{6,6,5,3,2,1,1}
\end{displaymath}
The following observation illustrates the importance of the rectangular ideals.

\begin{proposition} \label{prop:corners}
Let $\cC$ be the set of corners of $\lambda$. Then $I_{\lambda}=\bigcap_{(r,s) \in \cC} I_{r,s}$.
\end{proposition}

\begin{proof}
Let $\mu$ be a partition. We have
\begin{align*}
\bS_{2\mu} \subset \bigcap\nolimits_{(r,s) \in \cC} I_{r,s}
&\iff \rho(r,s) \subset \mu \ \text{for all $(r,s) \in \cC$} \\[-4pt]
&\iff \bigcup\nolimits_{(r,s) \in \cC} \rho(r,s) \subset \mu \\
&\iff \lambda \subset \mu \\
&\iff \bS_{2\mu} \subset I_{\mu}.
\end{align*}
The first and last step follow from Proposition~\ref{prop:sym2ideals}, the second is trivial, and the third is simply the observation that $\lambda=\bigcup_{(r,s) \in \cC} \rho(r,s)$. The result thus follows.
\end{proof}

\subsection{The variety $X$}

For $V=E \oplus F[1]$, let $B(V) = \Sym(\Sym^2(E) \oplus \lw^2(F))$. Note that $B(V)=A(V)/\rad(A(V))$. We regard $B$ as a 2-variable tca (in the variables $E$ and $F$). Let $X(V)=\Spec(B(V))$, which we identify with $\Sym^2(E)^* \times \lw^2(F)^*$. By a ``closed subvariety'' of $X$, we mean a subfunctor $Y$ of $X$ such that $Y(V)$ is a closed subvariety of $X(V)$ for all finite dimensional $V$. Closed subvarieties of $X$ correspond bijectively to $\GL \times \GL$ stable radical ideals of $B$. For $r,s \in \bN \cup \{\infty\}$, let $X_{r,s}(V) \subset X(V)$ be the locus of pairs $(\omega, \eta)$ such that $\rank(\omega) \le r$ and $\rank(\eta) \le 2s$. Then $X_{r,s}$ is a closed subvariety of $X$ in the above sense. We now show that these account for essentially all examples:

\begin{proposition}
Let $Y$ be a closed subvariety of $X$. Then there is a finite subset $\cC$ of $(\bN \cup \{\infty\})^2$ such that $Y=\bigcup_{(r,s) \in \cC} X_{r,s}$.
\end{proposition}

\begin{proof}
By the \emph{rank} of a point $(\omega, \eta) \in X(E,F)$, we mean the pair $(\rank(\omega), \tfrac{1}{2} \rank(\eta))$. Let $S \subset \bN^2$ be the set of pairs $(r,s)$ such that $Y(E,F)$ has a point of rank $(r,s)$ for some $E$ and $F$. We claim that $(r,s) \in S$ if and only if $X_{r,s} \subset Y$. It is clear that $X_{r,s} \subset Y$ implies $(r,s) \in S$. Conversely, suppose that $(r,s) \in S$. Then there exists some $(\omega_0, \eta_0) \in Y(E_0,F_0)$ of rank $(r,s)$ for some $E_0$ and $F_0$. Let $(\omega, \eta) \in X_{r,s}(E,F)$ be given. By basic linear algebra, there are linear maps $\phi \colon E \to E_0$ and $\psi \colon F \to F_0$ such that $\omega=\phi^*(\omega_0)$ and $\eta=\psi^*(\eta_0)$. Thus the map $X(E_0,F_0) \to X(E,F)$ defined by $(\phi, \psi)$ carries $(\omega_0, \eta_0)$ to $(\omega, \eta)$. Since $(\omega_0, \eta_0) \in Y(E_0,F_0)$ and $Y$ is a subfunctor of $X$, it follows that $(\omega, \eta) \in Y(E,F)$. This proves the claim.

It now follows that if $(r,s) \in S$ and $(r',s') \le (r,s)$ then $(r',s') \in S$, where here $(r',s') \le (r,s)$ means $r' \le r$ and $s' \le s$. A simple combinatorial argument now shows that there is a finite subset $\cC$ of $(\bN \cup \{\infty\})^2$ such that $(r',s') \in S$ if and only if $(r',s') \le (r,s)$ for some $(r,s) \in \cC$. It follows that $Y$ is the union of the $X_{r,s}$ with $(r,s) \in \cC$, which proves the proposition.
\end{proof}

\begin{corollary} \label{cor:irredX}
Any irreducible closed subvariety of $X$ is one of the $X_{r,s}$.
\end{corollary}

\subsection{The vanishing locus of $I_{r,s}$}

The goal of this section is to prove the following:

\begin{proposition}
We have $\cV(I_{r+1,s+1}(V))=X_{r,s}(V)$ for $r,s \ge 0$.
\end{proposition}

We break the proof into two lemmas.

\begin{lemma}
We have the following:
\begin{enumerate}
\item We have $\rad(I_{r+1,s+1}(E))=I_{r+1,1}(E)$ in $\Sym(\Sym^2(E))$.
\item We have $\rad(I_{r+1,s+1}(F[1]))=I_{1,s+1}(F)$ in $\Sym(\lw^2(F))$.
\item We have $\cV(I_{r+1,s+1}(V)) \subset X_{r,s}(V)$.
\end{enumerate}
\end{lemma}

\begin{proof}
(i) This is proved in \cite{abeasis}, but we include an argument (due to Steven Sam) to be self-contained. By Proposition~\ref{prop:sym2ideals}, we have $I_{r+1,s+1}(E) \subset I_{r+1,1}(E)$. We thus have a surjection $\pi \colon A(E)/I_{r+1,s+1}(E) \to A(E)/I_{r+1,1}(E)$. By Proposition~\ref{prop:sym2ideals}, we have
\begin{align*}
A(E)/I_{r+1,1}(E) &= \bigoplus_{\ell(\lambda) \le r} \bS_{2\lambda}(E), \\
A(E)/I_{r+1,s+1}(E) &= \bigoplus_{\ell(\lambda) \le r} \bS_{2\lambda}(E) + \bigoplus_{w(\lambda) \le s} \bS_{2\lambda}(E)
\end{align*}
We thus see that $\ker(\pi)$ is the sum of those $\bS_{\lambda}(E)$'s with $w(\lambda) \le s$ and $\ell(\lambda)>r$. However, $\bS_{\lambda}(E)=0$ if $\ell(\lambda)>\dim(E)$, and so there are only finitely many relevant such $\lambda$. Thus $\ker(\pi)$ is finite dimensional, and therefore nilpotent (since it is homogeneous and consists of positive degree elements), and so the claim follows.

(ii) This is proved in \cite[Theorem~5.1]{pfaffians}. We can also argue analogously to the  above.

(iii) Since $I_{r+1,s+1}(V)$ contains both $I_{r+1,s+1}(E)$ and $I_{r+1,s+1}(F[1])$, we see that its radical contains both $I_{r+1,1}(E)$ and $I_{1,s+1}(F[1])$ by (i) and (ii), and thus the extensions of these ideals to $A(V)$. It follows that $\cV(I_{r+1,s+1}(V))$ is contained in the intersection of $\cV(I_{r+1,1}(E)^e)$ and $\cV(I_{1,s+1}(F[1])^e)$. Now, $I_{r+1,1}(E)$ is the classical determinantal ideal: its vanishing locus in $\Sym^2(E)^*$ consists of those forms rank $\le r$. The vanishing locus of $I_{r+1,1}(E)^e$ is thus $X_{r,\infty}(V)$. Similarly, the vanishing locus of $I_{1,s+1}(F[1])^e$ is $X_{\infty,s}(V)$. We thus find that $\cV(I_{r+1,s+1}(V))$ is contained in $X_{r,\infty}(V) \cap X_{\infty,s}(V)=X_{r,s}(V)$.
\end{proof}

\begin{lemma}
We have $X_{r,s}(V) \subset V(I_{r+1,s+1}(V))$.
\end{lemma}

\begin{proof}
Let $\alpha$ be a non-degenerate symmetric bilinear form on $\bC^r$. We obtain a natural map
\begin{displaymath}
\Sym^2(E) \to \Sym^2(\bC^r) \otimes \Sym^2(E) \to \Sym^2(\bC^r \otimes E),
\end{displaymath}
where the first map comes from the inclusion $\bC \to \Sym^2(\bC^r)$ of the $\bO_r$ invariant provided by $\alpha$, and the second map comes from the Cauchy decomposition. The above map induces an algebra homomorphism $f \colon \Sym(\Sym^2(E)) \to \Sym(\bC^r \otimes E)$. The induced map on spectra $f^* \colon \Hom(E, \bC^r) \to \Sym^2(E)^*$ takes a linear map $\phi \colon E \to \bC^r$ to the form $\phi^*(\alpha)$ on $E$, and thus surjects onto the locus of forms of rank $\le r$.

Now let $\beta$ be a non-degenerate symplectic form on $\bC^{2s}$. A similar construction yields a homomorphism $g \colon \Sym(\lw^2(F)) \to \Sym(\bC^{2s} \otimes F)$ such that $g^*$ surjects onto the locus of forms in $\lw^2(F)^*$ of rank $\le 2s$.

Finally, let $\gamma$ be the non-degenerate orthosymplectic form on $\bC^{r|2s}$ that restricts to $\alpha$ and $\beta$ on the even and odd pieces. Once again, we get a natural algebra homomorphism $h \colon \Sym(\Sym^2(V)) \to \Sym(\bC^{r|2s} \otimes V)$. One easily verifies that the following square commutes:
\begin{displaymath}
\xymatrix{
\Sym(\Sym^2(V)) \ar[r] \ar[d]_h &
\Sym(\Sym^2(E)) \otimes \Sym(\lw^2(F)) \ar[d]^{f \otimes g} \\
\Sym(\bC^{r|2s} \otimes V) \ar[r] &
\Sym(\bC^r \otimes E) \otimes \Sym(\bC^{2s} \otimes F) }
\end{displaymath}
Here the horizontal maps are the surjection of the left ring onto the quotient by its nilradical. It follows that $h^*$ surjects onto $X_{r,s}(V)$.

Finally, observe that the multiplicity space of $\bS_{\rho(r+1,s+1)}(V)$ in the algebra $\Sym(\bC^{r|2s} \otimes V)$ is $\bS_{\rho(r+1,s+1)}(\bC^{r|2s})$ by the Cauchy decomposition, which vanishes by Proposition~\ref{prop:schur-eval}. Thus $\ker(h)$ contains $I_{r+1,s+1}(V)$, and so $\cV(I_{r+1,s+1}(V))$ contains the image of $h^*$. This proves the lemma.
\end{proof}

\subsection{The main theorem}

We now come to our main result:

\begin{theorem} \label{thm:sym2}
The $\GL$-primes of $A$ are exactly the ideals $I_{r,s}$ with $r,s \ge 1$ and the zero ideal. The $\GL$-radical ideals of $A$ are exactly the ideals $I_{\lambda}$, and the zero and unit ideal.
\end{theorem}

\begin{corollary}
Let $S = \bN^2 \cup \{\infty\}$ equipped with the partial order described as follows: $(r,s)<\infty$ for all $(r,s)$; and $(r,s) \le (r',s')$ if $r \le r'$ and $s \le s'$. Endow $S$ with the unique sober topology for which $\le$ is the generalization order on points. Then $\Spec_{\GL}(A)$ is homeomorphic to $S$.
\end{corollary}

Before proving the theorem, we require a lemma.

\begin{lemma}
Suppose $\dim(E) \ge r$ and $\dim(F) \ge 2s$. Let $\lambda$ be a partition such that $\rho(r,s) \not\subset \lambda$. Then $\cV(I_{r,s}(V)) \not\subset \cV(I_{\lambda}(V))$.
\end{lemma}

\begin{proof}
By Proposition~\ref{prop:corners}, we have $\cV(I_{\lambda}(V)) = \bigcup_{(p,q) \in \cC} \cV(I_{p,q}(V))$ where $\cC$ is the set of corners of $\lambda$. For any $(p,q) \in \cC$ we have $\rho(r,s) \not\subset \rho(p,q)$, that is, $r>p$ or $s>q$. It follows that $V(I_{p,q}(V))$ does not contain any pair $(\omega, \eta)$ with $\rank(\omega)=r-1$ and $\rank(\eta)=2s-2$. However, $\cV(I_{r,s}(V))$ does contain such pairs.
\end{proof}

\begin{proof}[Proof of Theorem~\ref{thm:sym2}]
Since $\cV(I_{r,s}(V))=X_{r-1,s-1}(V)$ is irreducible for all super vector spaces $V$, we see that $\rad_{\GL}(I_{r,s})$ is $\GL$-prime by Theorem~\ref{thm2}. If $I$ is a $\GL$-ideal of $A$ that properly contains $I_{r,s}$ then it contains some $\bS_{2\lambda}$ with $\rho(r,s) \not\subset \lambda$, and then $\cV(I_{r,s}(V)) \not\subset \cV(I(V))$ for large $V$ by the lemma. Since $\cV(I_{r,s}(V))=\cV(\rad_{\GL}(I_{r,s})(V))$ for all $V$ by Theorem~\ref{thm1}, it follows that $I_{r,s}=\rad_{\GL}(I_{r,s})$. Thus $I_{r,s}$ is $\GL$-prime.

We now show that the $I_{r,s}$ account for all the non-zero $\GL$-primes of $A$. Thus let $I$ be some non-zero $\GL$-prime ideal of $A$. Then $\cV(I(V))$ is irreducible for all $V$ by Theorem~\ref{thm1}. The rule $V \mapsto \cV(I(V))$ defines an irreducible closed subvariety of $X$, and therefore coincides with $X_{r,s}$ for some $r,s \in \bN \cup \{\infty\}$ by Corollary~\ref{cor:irredX}. Since $I$ contains $I_{\rho(r',s')}$ for some $r'$ and $s'$, it follows that $r,s<\infty$. Thus $\rad(I(V))=\rad(I_{r+1,s+1}(V))$ for all $V$, and so $I=I_{r+1,s+1}$ by Theorem~\ref{thm1}, since both $I$ and $I_{r+1,s+1}$ are $\GL$-radical.

Since $I_{\lambda}$ is an intersection of rectangular ideals (Proposition~\ref{prop:corners}), it is therefore $\GL$-radical. An argument similar to the one in the proof of Proposition~\ref{prop:corners} shows that any intersection of rectangular ideals is equal to $I_{\lambda}$ for some $\lambda$, or the zero or unit ideal. Since any $\GL$-radical ideal is an intersection of $\GL$-primes, the result follows.
\end{proof}

\begin{remark}
The idea that Theorem~\ref{thm:sym2} should be true came out of joint work with Steven Sam.
\end{remark}

\begin{remark}
There is an alternate method for proving that $I_{r,s}$ is $\GL$-prime: explicitly compute the product ideal $I_{\lambda} I_{\mu}$ in $A$, for all $\lambda$ and $\mu$, and verify the primality condition directly. As far as we know, the computation of $I_{\lambda} I_{\mu}$ does not appear in the literature in this case. However, a closely related case (namely, that of $\Sym(V \otimes W)$ with $\GL(V) \times \GL(W)$ acting) is treated in \cite{whitehead}.
\end{remark}

\end{document}